\documentclass[12pt, a4paper, reqno]{amsart}
\usepackage[utf8]{inputenc}

\usepackage{color}

\usepackage[margin=1.1in,top=1.4in,bottom=1.4in]{geometry}
\usepackage{amsmath, amsthm, amssymb, amsfonts}
\usepackage{calc}
\usepackage{url}
\usepackage{enumitem}

\usepackage{cleveref}
\crefformat{section}{\S~#2#1#3} 
\crefformat{subsection}{\S~#2#1#3}
\crefformat{subsubsection}{\S~#2#1#3}

\numberwithin{equation}{section}
\newtheorem{theorem}{Theorem}[section]
\newtheorem{corollary}[theorem]{Corollary}
\newtheorem{lemma}[theorem]{Lemma}
\newtheorem{proposition}[theorem]{Proposition}

\theoremstyle{definition}

\title{Lower bounds in the polynomial Szemer\'edi theorem}

\author{Khalid Younis}
\address{Khalid Younis, South Yorkshire, United Kingdom}
\email{younis.maths@outlook.com}

\begin{document}
\begin{abstract}
We construct large subsets of the first $N$ positive integers which avoid certain arithmetic configurations. In particular, we construct a set of order $N^{0.7685}$ lacking the configuration $\{x,x+y,x+y^2\},$ surpassing the  $N^{3/4}$ limit of Ruzsa's construction for sets lacking a square difference. We also extend Ruzsa's construction to sets lacking polynomial differences for a wide class of univariate polynomials. Finally, we turn to multivariate differences, constructing a set of order $N^{1/2}$ lacking a difference equal to a sum of two squares. This is in contrast to the analogous problem of sets lacking a difference equal to a prime minus one, where the current record is of order $N^{o(1)}.$
\end{abstract}

\maketitle

\section{Introduction}
The Furstenberg-Sárközy theorem \cite{Furst, Sark}  states that if a subset \(A\subseteq[N]:=\{1,2,\dots,N\}\) has \emph{difference set}  \(A-A:=\{a-a':a,a'\in A\}\) disjoint from the non-zero squares, then \(\vert A \vert =o(N)\).  This result has been improved and generalised in many directions. The current record quantitative bound is due to Pintz,  Steiger, and Szemerédi \cite{Pintz88}.  Replacing squares by a univariate polynomial with integer coefficients, a general criteria for such a result was determined by Kamae and Mend\`es France \cite{KamMend}, with bounds analogous to \cite{Pintz88} due to Rice \cite{Rice19}. Rice \cite{Rice18} has also considered difference sets avoiding binary forms. In all of these works there is little recorded regarding corresponding lower bounds, apart from the case of pure-power differences considered by Ruzsa \cite{Ruzsa84}.

The polynomial Szemer\'edi theorem \cite{BergLei} is a deep generalisation of these results, albeit qualitative in nature.  There has been much recent work on quantitative upper bounds in this theorem, with Green \cite{Green} considering sets lacking three-term progressions with common difference equal to a sum of two squares, Prendiville \cite{Prend17} tackling arbitrarily long progressions with difference equal to a perfect power, and with Peluse and Prendiville \cite{PelPrend19} dealing with the non-linear Roth configuration \begin{equation}\label{nonlinear roth}
\{x,x+y,x+y^2\}.
\end{equation}  
For the former two configurations, the Behrend construction \cite{Behr} provides a lower bound of super-polynomial size. However, this is not applicable to the non-linear Roth configuration \eqref{nonlinear roth}, and it has been speculated that the correct order of magnitude for this problem may be closer to  that of the Furstenberg--S\'ark\"ozy problem.

The goal of this paper is threefold: to construct a set lacking the non-linear Roth configuration \eqref{nonlinear roth} which is larger than the current record lower bound in the Furstenberg--S\'ark\"ozy theorem; to extend Ruzsa's construction \cite{Ruzsa84} beyond the case of perfect power differences; and to obtain the first polynomial lower bounds for sets lacking certain multivariate differences. Here then are three representative results.
 \begin{theorem}\label{nonlincor}
For all \(\varepsilon>0\) and all positive integers \(N,\) there exists a set \(A \subseteq [N]\) with no non-trivial\footnote{Non-trivial meaning \(y\in \mathbb{Z}\setminus \{0\}.\)} configuration of the form \(\{x,x+y,x+y^2\}\) and \(\vert A \vert \gg_{\varepsilon} N^{\gamma - \varepsilon},\) where
 \[\gamma = \frac{1}{2}+\frac{\log_{65}7}{3}+\frac{\log_{65}{17}}{6} =0.7685\dots.\]
 \end{theorem}
 
\begin{theorem}\label{expoly}
For all positive integers \(N\), there exists a set \(A\subseteq[N]\) with all non-zero differences avoiding the set \(\{x^2+5x^3:x\in \mathbb{Z}\}\) and \(\vert A \vert \gg N^{\gamma},\) where
\[\gamma = \frac{2+\log_52}{3}=0.8102 \dots.\]
\end{theorem}

\begin{theorem}\label{sumsquare} For all positive integers \(N,\) the following hold. 
\begin{enumerate}[label=\upshape(\roman*), widest=iii]
    \item There exists a set \(A\subseteq[N]\) of size \(\vert A \vert \gg N^{1/2}\) with all non-zero differences avoiding the set \(\{x^2+y^2: x,y\in \mathbb{Z}\}.\) \label{i}
    \item There exists a set \(A\subseteq[N]\) of size \(\vert A \vert \gg N^{1/4}\) with all non-zero differences avoiding the set \(\{x_1^4+\dots +x_7^4: x_1,\dots, x_7\in \mathbb{Z}\}.\) \label{ii}
    
\end{enumerate}
\end{theorem}

 These are applications of the more general Theorems \ref{config}, \ref{polygeneralised} and \ref{general2} below.

 It is important to contrast our result with those available from density arguments alone. Given \(B \subseteq [N]\), the greedy algorithm (see \cite[\S~B.2]{Lyal13}) delivers a set \(A\subseteq [N]\) whose non-zero differences are not in \(B\) and where \(\vert A\vert \gg N/\vert B\vert.\) For example, an immediate lower bound for the problem in Theorem \ref{expoly} is a constant times \(N^{2/3}.\)  Similarly, one has
 \[\vert\{x^2+y^2\leq N:x
 \in \mathbb{Z}\}\vert \sim K\frac{N}{\sqrt{\log N}}\] where \(K\approx 0.764\) is the Landau-Ramanujan constant \cite[\S~2.3]{mathconst}. A greedy argument would therefore only deliver \(\vert A \vert\gg\sqrt{\log N}\) for Theorem  \ref{sumsquare} \ref{i}. 
 
 There is a well-used analogy between sums of two squares and primes minus one.  Indeed S\'ark\"ozy has shown that sets of integers lacking a difference equal to a prime minus one cannot have positive density, with the current quantitative record due to Wang \cite{Wang}. At present, the lower bounds for this problem are not too far from the greedy construction, of order $N^{o(1)}$ (see Ruzsa \cite{Ruzprime}). In relation to this problem, one may  view Theorem \ref{sumsquare} \ref{i} either as evidence towards the existence of a polynomial lower bound for primes minus one, or as evidence towards a demarcation in  the analogy with sums of two squares.
 
 In order to surpass the greedy bound in our results, we exploit the arithmetic structure of the set of differences to be avoided. The initial observation is that certain sets can be obstructed by congruence arguments; such as how squares are never congruent to \(\pm2\) modulo \(5.\)
 This was exploited in the construction of Ruzsa \cite{Ruzsa84}  for differences avoiding \(k\)th powers. We build upon this in several ways. 

 For fixed integers \(m\) and \(k\geq 2,\)  let \(R\subseteq\{0,1,\dots, m-1\}\) denote a subset whose differences between distinct elements never equal a \(k\)th power modulo \(m.\) For example, \(R=\{0,2\}\) has all non-zero differences avoiding squares modulo \(5.\)
 
\begin{theorem}[Ruzsa]\label{ruz}
Let \(m\) be square-free and \(k\geq2.\) For all positive integers \(N,\) there exists a set \(A \subseteq [N]\) with \((A-A) \cap \{x^k: x \in \mathbb{Z}\} = \{0\}\) and \(\vert A \vert \gg_m N^\gamma,\) where \begin{equation*}
\gamma =\frac{k-1+\log_m\vert R \vert}{k}.\end{equation*}
\end{theorem}
One can try to find optimal \(m\) and \(R\) to maximise this exponent for a given \(k.\) Indeed, the current best bound for square-free differences of \(\vert A \vert \gg N^{0.7334\dots}\) was established in this way by Beigel and Gasarch \cite{BeGa08} and independently by Lewko \cite{Lewko15}, both via a computer search. Likewise for cube-free differences, a calculation of Lewko \cite{Lewko15} established a lower bound of \(\vert A \vert \gg N^{0.8616\dots}\) (compare this to Theorem \ref{expoly}).

It was conjectured by Ruzsa that for \(k=2,\) one has \(\vert R \vert <m^{1/2}\) and thus  \(\gamma<3/4.\)  Ruzsa \cite{Ruzsa84} claimed to have proved this when \(m\) is a product of primes congruent to \(1\) modulo \(4.\) Based on this apparent limit of Ruzsa's construction, as well as a finite field analogy (see \cite[\S~1.4]{Rice19}), some have speculated that \(N^{3/4}\) may be the correct order of magnitude for the \emph{upper bound} as well. Observe that to avoid \(\{x,x+y,x+y^2\},\) one could simply take a set with  square-free differences. We improve upon this using the extra information from the \(x+y\) term --- in particular, we break through the \(N^{3/4}\) barrier.

 The key to Theorem \ref{nonlincor} is to obstruct configurations by controlling iterative difference sets: let \((R_n)_{n= 0}^\infty\) denote a sequence of non-empty subsets of \(\{0,1,\dots,m-1\}\) satisfying
 \((R_{n+1}-R_{n+1})\cap f(R_n-R_n) \subseteq\{0\}\) modulo \(m.\) When \(f(x)=x^k\) this means 
\[r_{n+1}-r_{n+1}'\not\equiv (r_n-r_n')^k \pmod m\]
for all distinct \(r_{n+1},r_{n+1}'\in R_{n+1}\) and all \(r_n,r_n' \in R_n.\) For instance, such a sequence could begin with \(R_0= \{0,1,\dots,m-1\}\) and \(R_1=R\) with \(R\) as previously defined.
 
\begin{theorem}[Non-linear progressions]\label{config}
   Let \(m\) be square-free and \(k\geq 2.\) For all \(\varepsilon>0\) and all positive integers \(N,\)  there exists a set \(A\subseteq[N]\) with no non-trivial configuration of the form \(\{x,x+y,x+y^k\}\)  and \(\vert A \vert\gg_{m,\varepsilon} N^{\gamma-\varepsilon},\) where 

\begin{equation}\label{lambda}
\gamma = (k-1)\left(\frac{\log_m\vert R_0\vert}{k}+\frac{\log_m\vert R_1 \vert}{k^2}+  \dots +\frac{\log_m\vert R_n \vert}{k^{n+1}} +  \dots\right),
\end{equation}
for sets $R_n$ defined as above.
\end{theorem}
 
\begin{theorem}[Inhomogeneous polynomials]\label{polygeneralised}
Let \(m\geq 2\) be square-free, \(d\geq k \geq 2, a_d\neq0\) and \(\gcd(a_k,m)=1\). Suppose that \(f(x)=\sum_{i=k}^d a_ix^i \in \mathbb{Z}[x]\) has zero as its only root modulo \(m.\) Then for all positive integers \(N,\) there exists a set \(A\subseteq[N]\) with \((A-A) \cap f(\mathbb{Z})=\{0\} \) and \(\vert A \vert \gg_{m,f} N^{\gamma},\) where \begin{equation} \label{gamma}
\gamma = \frac{d-1+\log_m\vert R \vert}{d}. \end{equation}
\end{theorem}

 Exponent \eqref{gamma} incorporates local (\(m\)-adic) obstructions from the \(x^k\) term (implicit in the defintion of \(R\)), as well as global (\(\mathbb{R}\)) obstructions  from the degree \(d\) of \(f.\)  Also note that the polynomials to which the above applies include those of the form \(a_kx^k+mx^kg(x)\) and \(a_kx^k+x^k(x+1)\dots(x+m-1),\) for example. One can trivially do much better for polynomials which have no root modulo some positive integer --- such a polynomial is said to not be  \emph{intersective}. For example, \(x^2+1\) has no root modulo \(3\) and it  follows that one may take \(A=3\mathbb{N}\cap [N].\)
 
Finally, for the multivariate case it is beneficial to work modulo \(m^k.\) For fixed \(m,k\) and \(F\) as below, let \(R'\subseteq\{0,1,\dots, m^k-1\}\) denote a subset whose differences between distinct elements are never in the image \(F(\mathbb{Z}^n)\) modulo \(m^k.\) 
\begin{theorem}[Homogeneous multivariate polynomials]\label{general2}
Let \(m\geq2\) be a positive integer. Let \(F(\mathbf{x}) \in \mathbb{Z}[x_1,\dots,x_n]\) be a homogeneous polynomial of degree \(k\geq 2.\) Suppose that the only roots of \(F\) modulo \(m^k\) are congruent to \(\mathbf{0}\) modulo \(m.\) 
 Then for all integers \(N,\) there exists a set \(A\subseteq[N]\) with \((A-A) \cap F(\mathbb{Z}^n)=\{0\} \) and \(\vert A \vert \gg_{m,k} N^{\gamma},\) where
\begin{equation}\label{gammalatter}
    \gamma=\frac{\log_m\vert R' \vert }{k}.
\end{equation}
\end{theorem}

This paper is organised as follows. In \cref{sectionnonlin} we prove Theorem \ref{config} and deduce Theorem \ref{nonlincor} by a suitable choice of \((R_n)_{n=0}^\infty.\) In \cref{sectioninhom} we prove Theorem \ref{polygeneralised}.  Theorem \ref{expoly} follows almost immediately. In \cref{sectionhom} we prove Theorem \ref{general2} and deduce Theorem \ref{sumsquare} from suitable choices of \(R'.\) In \cref{mod} we briefly discuss a connection between Theorem \ref{ruz}  and Theorem \ref{general2} by looking at the modular formulation of the problem.  Finally, in \cref{sectionopen} we outline an open problem of interest. 

\subsection*{Notation}
For functions \(f(N)\) and \(g(N),\) write \(f=O(g)\) or \(f \ll g\) or \(g\gg f\) to denote that there exists an absolute constant \(C>0,\) which may change at each appearance, such that \(\vert f\vert \leq C g \)  for all $N\geq 1.$ Write \(f\ll_\alpha g\) when \(C\) depends on \(\alpha\).  We write \(f\sim g\) to mean \(f/g \to 1\) and write \(f=o(g)\) to mean \(f/g\to 0,\) both as \(N \to \infty.\)

The floor function and ceiling function of \(x\)  are denoted \(\lfloor x \rfloor := \max\{n\in \mathbb{Z}:n\leq x\}\) and \(\lceil x \rceil:= \min\{n\in \mathbb{Z}:n\geq x\}\)  respectively.

\subsection*{Acknowledgement} The author would like to thank Sean Prendiville for suggesting the problem, his support and his helpful comments on earlier drafts.

 \section{Non-linear progressions}\label{sectionnonlin}

\begin{proof}[Proof of Theorem \ref{config}]
Consider the set of non-negative integers \(A\) with \(m\)-ary expansion \(\sum_{0 \leq i< Y}u_im^i\) satisfying
\begin{equation}\label{inhom2}
u_i \in \begin{cases}
R_0, &k\nmid i\\
R_1,  & k\mid\mid i \textrm{ or } i=0 \\
R_2, & k^2\mid\mid i \ \\
R_3, & k^3\mid\mid i \\
\vdots
\end{cases}
\end{equation}
where \(k^n \mid \mid i\) means \(k^n \mid i\) and \(k^{n+1} \nmid i.\) 

For elements \(u=\sum u_im^i, v=\sum v_im^i\) and \(w=\sum w_im^i\) in \(A,\) suppose for contradiction that \((u, v,w)=(u,u+x,u+x^k)\) with \(x\neq 0.\) Let \(j\) be the smallest index such that \(u_j\neq w_j,\) and \(\ell\) be the smallest index such that \(u_\ell\neq v_\ell.\)  We have
\[w-u=(v-u)^k,\]
thus 
\[m^j(w_j-u_j)+zm^{j+1}= (m^{\ell}(v_{\ell}-u_\ell)+\tilde{z}m^{\ell+1})^k\]
for some integers \(z\) and \(\tilde{z}.\) Looking at the highest power of \(m\) which divides both sides gives \(j=k\ell\) (using that \(m\) is square-free), so in particular \(j\) is not in the first case of \eqref{inhom2}. On dividing by \(m^j\) we find that
\[w_j-u_j\equiv(v_{j/k}-u_{j/k})^k \pmod m.\] 
By the definition of \((R_n)_{n=0}^\infty\) and the construction in \eqref{inhom2}, this is impossible. 

Now  we calculate the size of \(A.\)  The number of multiples of \(k^n\) in \([1,Y)\cap \mathbb{N}\)  is 
\(\lceil Y/k^n\rceil-1,\) and therefore 
\[\vert\{i: 1\leq i < Y \text{ and } k^n \mid\mid i\} \vert = \left\lceil\frac{Y}{k^n}\right\rceil-\left\lceil\frac{Y}{k^{n+1}}\right\rceil .\] It follows that for all \(n\) we have 
\begin{align*}
\vert A \vert & \geq \vert R_0 \vert ^{\lceil{Y\rceil}-\lceil{Y/k}\rceil}\vert R_1 \vert ^{1+\lceil {Y/k}\rceil-\lceil{Y/k^2}\rceil} \vert R_2 \vert^{\lceil{Y/k^2}\rceil-\lceil{Y/k^3}\rceil}\dots \vert R_{n-1} \vert ^{\lceil{Y/k^{n-1}}\rceil-\lceil{Y/k^{n}}\rceil}\\
&\geq m^{1-n} \vert R_0 \vert ^{(k-1)Y/k}\vert R_1 \vert ^{(k-1)Y/k^2} \vert R_2 \vert^{(k-1)Y/k^3}\dotsm \vert R_{n-1} \vert ^{(k-1)Y/k^n}\\
& = m^{1-n} \left(|R_0| |R_1|^{1/k} \dotsm |R_{n-1}|^{1/k^{n-1}}\right)^{(1-1/k)Y}.
\end{align*}

Fix \(\varepsilon>0.\)  Choose \(n=n(\varepsilon)\) sufficiently large so that
$$
\varepsilon \geq \sum_{j\geq n} 2^{-j} \geq \sum_{j \geq n} \frac{\log_m|R_j|}{k^j}.
$$ 
Assume that \(N>m.\) Then taking $Y:= \log_m N $ gives \(\max A <m^Y =N,\) thus \(A \subseteq \{0,1,\dots,N-1\}\) and 
\begin{align*}
|A| \geq  m^{1-n}\left(|R_0| |R_1|^{1/k} \dotsm |R_{n-1}|^{1/k^{n-1}}\right)^{(1-1/k)\log_mN} \geq m^{1-n} N^{\gamma - \varepsilon}.
\end{align*}
Finally, add \(1\) to all the elements of \(A.\)

When \(1\leq N \leq m\) we can simply take \(A\) to be a singleton. We may adjust the implicit constant so that $\vert A \vert  \gg_{m, \varepsilon} N^{\gamma - \varepsilon}$ then holds for all \(N\geq 1\) as desired. 
\end{proof}

 \begin{proof}[Proof of Theorem \ref{nonlincor}.]
Choosing \((R_n)_{n=0}^\infty=(\{0,1,\dots, m-1\}, R_1, R_2, R_1, R_2, \dots)\)  and evaluating \(\eqref{lambda}\) gives
\[\gamma=\frac{k-1}{k}+ \frac{\log_m\vert R_1 \vert}{k+1}+ \frac{\log_m\vert R_2 \vert}{k(k+1)}.\]

We may verify\footnote{Note that \(R_1\) and \(R_2\) were found by adapting a maximum-clique-searching C\texttt{++} algorithm by Konc \cite{koncweb} based on the work of Konc and  Janežić \cite{konc}.} that for \( k=2,  m=65,\) the sets \[ 
R_1=\{31,39,8,62,19,42,50\}
\]
and \[
R_2=\{31,47,62,34,42,39,27,8,54,23,0,58,19,50,15,12,4\}
\]  satisfy the conditions of Theorem \ref{config}.
 \end{proof}
 
\section{Inhomogeneous polynomials}\label{sectioninhom}

We begin with a small lemma which roughly states that if many square-free \(m\) divide the polynomial, then many \(m\) must divide the variable. 
\begin{lemma} \label{divide}
For \(m\) and \(f\) given in Theorem \ref{polygeneralised}, if \(f(x)\equiv 0 \pmod {m^j}\) then \(x \equiv 0 \pmod {m^{\lceil{j/k}\rceil}}.\)
\end{lemma}
\begin{proof}
If \(j\leq k\) then the statement holds by the condition of the theorem. Otherwise, certainly \(f(x)\equiv 0 \pmod {m},\) thus \(x \equiv 0 \pmod m\) by the given property of \(f.\) Say \(x=my.\) Thus \(\sum_{i=k}^d a_im^iy^i \equiv 0 \pmod {m^j}.\) So \(g(y) := \sum_{i=k}^d a_im^{i-k}y^i \equiv 0 \pmod {m^{j-k}}.\) Then we can check \(g(y)\) satisfies the condition of the theorem too (using that \(m\) is square-free). Inductively, the result follows.
\end{proof}

The key ingredient in the proof of Theorem \ref{polygeneralised} is that once we have forced the variable \(x\) to have enough factors of \(m,\) we can then use properties of \(\mathbb{R}\) to obstruct solutions --- namely, that differences between two elements in \([N]\) are always less than \(N.\)

\begin{proof}[Proof of Theorem \ref{polygeneralised}]
Let us first deal with the case of \(a_k=1.\) 

For real numbers \(X\) and \(Y\) to be chosen later, consider the set of integers \(A\) with \(m\)-ary expansion \(\sum_{0 \leq i< Y}u_im^i\) satisfying

\begin{equation}\label{digits}
u_i \in \begin{cases}
R, & 0\leq i< X  \textrm{ and }  k\mid i ;\\
\{0,1,\dots,m-1\}, & \text{otherwise.}
\end{cases}
\end{equation}

For elements \(u=\sum u_im^i\) and \(v=\sum v_im^i\) in \(A\),  suppose for contradiction that \(u-v\in f(\mathbb{Z}) \setminus \{0\}.\) Let \(j\) denote the smallest index such that \(u_j \neq v_j.\) Then \[f(x)=u-v = m^j(u_j-v_j)+  z{m^{j+1}}\] for some integers \(z\) and $x$ with $x \neq 0$. By Lemma \ref{divide} we have that \(x=m^{\lceil{j/k}\rceil}y\) for some \(y.\) Thus
\begin{equation*}
    \sum_{i=k}^d a_i(m^{\lceil{j/k}\rceil}y)^i = m^j(u_j-v_j)+  z{m^{j+1}}.
\end{equation*}
This can be factorised into 
\begin{equation}\label{poly}
    m^{\lceil{j/k}\rceil k}y^k(1 + mg(y))= m^j(u_j-v_j)+  z{m^{j+1}}
\end{equation}
for some polynomial \(g \in \mathbb{Z}[y].\)
There are three cases to consider.

\emph{Case 1.} Suppose that $0\leq j< X$ and $k\mid j$.  Since \(\lceil{j/k}\rceil=j/k\), we see that \(m^j\) divides the left-hand-side of \eqref{poly}. On dividing by \(m^j\) we find that
\begin{equation}\label{a_k}
 y^k\equiv u_j-v_j \pmod m,    
\end{equation}
which contradicts the definition of \(R.\)

\emph{Case 2.} Suppose that $0\leq j< X$ and $k\nmid j$.
The greatest  power of \(m\) which divides the right-hand-side of \eqref{poly} is \(j.\) However, since \(m\) is square-free, we see that the greatest power of \(m\) dividing the left-hand-side is a multiple of \(k.\) This implies \(k\mid j,\) which is a contradiction. 

\emph{Case 3.} Suppose that $X \leq j < Y$.  Observe that $|u-v|$ lies in \([0,m^Y)\). Also, the polynomial \(f\) is dominated by its leading term, hence if \(X\) (and hence $j$) is sufficiently large in terms of the coefficients and degree of $f$, we have  for all non-zero integers \(y\) that
\[\frac{\vert a_d\vert}{2}m^{jd/k}\leq \frac{\vert a_d\vert}{2}m^{\lceil{j/k}\rceil d}\vert y \vert^d  < \vert f(m^{\lceil{j/k}\rceil}y)\vert = \vert u-v\vert  < m^Y.\]
Set\footnote{Of course if \(k=d\) then simply set \(X:=Y,\) then we cover all possibilities in cases 1 and 2 already} \(X:=k(Y+1)/d\) with $Y$ sufficiently large in terms of the coefficients and degree of $f$, so that
\[\frac{\vert a_d\vert}{2}m^{Y+1}< m^Y.\] This leads to \(m<2/\vert a_d \vert \leq 2,\) which is a contradiction. Hence no such \(u\) and \(v\) exist for this choice of $X$ and $Y$.

 All that is left is to calculate the size of this set \(A\). For \(\gamma\) as in \eqref{gamma}, this is
 \[\vert A \vert \geq {\vert R \vert}^{\lceil{X/k}\rceil}m^{\lceil{Y\rceil}-\lceil{X/k}\rceil}\geq m^{-2} \vert  R\vert ^{Y/d}m^{Y-Y/d}\gg_m m^{Y\gamma}.\]

Assume \(N\) is sufficiently large in terms of $m$ and the polynomial $f$ such that the quantity \(Y:= \log_m N\) is sufficiently large for the above argument to hold. Then  $\vert A \vert  \gg_{m} N^{\gamma}.$ Since \(\max A <m^Y=N,\) we have
\(A \subseteq \{0,1,\dots, N-1\}.\) Finally, add \(1\) to all elements of \(A.\) 

Therefore, we have \(N\ll_{m,f}1\) or $\vert A \vert  \gg_{m} N^{\gamma}.$ In the former case, we may take \(A\) to be a singleton. It follows that $\vert A \vert  \gg_{m, f} N^{\gamma}$ holds for all \(N\geq1\) as desired. 

We assumed throughout that \(a_k=1.\) For \(a_k\neq 1,\) we simply replace \(R\) with \(a_kR.\) Since we insist that \(\gcd (a_k,m)=1,\) we have \(\vert R \vert= \vert a_k R\vert.\) The only change in the argument is that \eqref{a_k} becomes 
\[a_ky^k\equiv u_j- v_j \pmod m\]
for \(u_j,v_j\in a_kR.\) Then \(u_j=a_k u_j'\) and \(v_j=a_kv_j'\) for some \(u_j',v_j' \in R,\) and so
\[y^k\equiv u_j'- v_j' \pmod m,\] giving the same contradiction as before.
\end{proof}
\begin{proof}[Proof of Theorem \ref{expoly}]
Choose \(m=5, R=\{0,2\}\) and \(f(x)=x^2+5x^3.\) Then \(f\) satisfies the conditions of the theorem: if \(f(x)\equiv 0 \pmod 5\) then \(x^2\equiv 0 \pmod 5\), thus \(x\equiv 0 \pmod 5,\) as \(5\) is prime. 
\end{proof}

\section{Homogeneous multivariate polynomials}\label{sectionhom}

\begin{proof}[Proof of Theorem \ref{general2}]
Let \(M=m^k.\) Consider the set of integers \(A\) with \(M\)-ary expansion \(\sum_{0 \leq i< Y}u_iM^i\) satisfying \(u_i \in R'\) for all \(i.\)

For elements \(u=\sum u_iM^i\) and \(v=\sum v_iM^i\) in \(A,\) suppose for contradiction that \(u-v \in F(\mathbb{Z}^n) \setminus \{0\}.\) Let \(j\) be the smallest index such that \(u_j\neq v_j.\) Then 
\[F(\mathbf{x})=u-v=M^j(u_j-v_j)+zM^{j+1}\]
for some integer \(z.\)
If \(j\geq1\) then by the property of \(F\) in the statement of the theorem we have \(\mathbf{x}=m\mathbf{y}.\) Then \(F(\mathbf{x})=m^kF(\mathbf{y})\) by homogeneity. Hence \[F(\mathbf{y})=M^{j-1}(u_j-v_j)+zM^{j}.\] Inductively this leads to \(F(\mathbf{z})=(u_j-v_j)+zM,\) where \(\mathbf{x}=m^j\mathbf{z},\) which contradicts the definition of \(R'.\) Hence no such \(u\) and \(v\) exist.

Now to calculate the size of \(A.\) That is,
\[\vert A \vert = \vert R' \vert ^{\lceil{Y\rceil}} \geq M^{Yk^{-1}\log_m\vert R' \vert}.\]

Assume that \(N>M.\) Then taking $Y:= \log_M N $ gives \(\max A <M^Y =N,\) thus \(A \subseteq \{0,1,\dots,N-1\}.\) Also, \(\vert A \vert \geq N^\gamma\) for \(\gamma\) as in \eqref{gammalatter}. Add \(1\) to all elements of \(A.\)

If \(1\leq N \leq M,\) then we may take \(A\) to be a singleton. By adjusting the implicit constant, we arrive at \(\vert A \vert \gg_{M} N^\gamma\) for all \(N\geq 1.\) Since \(M=m^k,\) the implicit constant depends only on \(m\) and \(k\) as desired.
\end{proof}

\begin{proof}[Proof of Theorem \ref{sumsquare} (i)]
Choose a prime \(p \equiv 3 \pmod 4.\) Then \(-1\) is not a quadratic residue modulo \(p.\)  We first show that \(x^2+y^2 \equiv 0 \pmod {p}\) if and only if \(x,y \equiv 0 \pmod p\) if and only if \(x^2+y^2 \equiv 0 \pmod {p^2}.\) All that needs to be shown is the first one-way implication. Indeed, if \(y \not\equiv 0 \pmod p\) then  \((x/y)^2\equiv -1 \pmod p,\) which is a contradiction. Thus \(y \equiv 0 \pmod p\) and necessarily \(x \equiv 0 \pmod p.\) In particular, \(F(x,y)=x^2+y^2\) satisfies the conditions of Theorem \ref{general2}.

Moreover, consider \(R'=\{0,p,\dots, (p-1)p\}\) of size \(p\) whose differences between distinct elements are multiples of \(p\) but not \(p^2.\) Thus non-zero  differences of \(R'\) never equal \(F(x,y)\) modulo \(p^2.\)
\end{proof}
This argument would attain the same bound \(\vert A \vert \gg_p N^{1/2}\) for the binary form \(x^2+y^2+pxy\) (again where \(p\equiv 3 \pmod 4\)). We may similarly show \(\vert A \vert\gg_k N^{1-1/k}\) for \(ax^k+by^k,\) so long as we can find a prime for which \(-a/b\) is not a \(k\)th power residue.
\begin{proof}[Proof of Theorem \ref{sumsquare} (ii)]
Observe that \(x^4\equiv 0 \textrm{ or } 1 \pmod {16}\) when \(x\) is even or odd respectively. Thus, if \(\sum_1^7 x_i^4 \equiv 0 \pmod {16}\) then all \(x_i\equiv 0 \pmod 2.\) Additionally, we may take \(R'= \{0,8\},\) since \(\sum_1^7 x_i^4 \not\equiv \pm 8 \pmod {16}.\)
\end{proof}
Especially in the multivariate scenario, we can extract more information from a modulo \(m^k\)  argument than from a modulo \(m\) argument. The above example illustrates this, since fourth powers are simple to analyse modulo \(16,\) but little would have been achieved by modulo \(2\) reasoning only.

\section{The Modular Problem}\label{mod}
The construction in \cref{sectioninhom} was base \(m,\) whereas the  construction of \cref{sectionhom} was base \(m^k.\) Both of these constructions can be applied to the problem of  differences avoiding  \(\{x^k:x\in \mathbb{Z}\}.\) We might expect that these constructions are equivalent. The aim of this section is demonstrate a partial result in this direction, showing a connection between the two constructions.

Recall that for fixed integers \(m\) and \(k\geq2,\) we write \(R\subseteq\{0,1,\dots, m-1\}\) for a subset whose non-zero differences are never equal to a \(k\)th power modulo \(m.\) Let \(r_k(m):=\max \vert R \vert\) and insist that \(m\) is square-free throughout this section.

For the task of constructing sets free of $k$th power differences, Theorem \ref{polygeneralised} (or equivalently Theorem \ref{ruz}) gives a lower bound of order \(N^\gamma,\) where
$$
\gamma=\frac{k-1+\log_m r_k(m)}{k},
$$
whilst Theorem \ref{general2} with \(F\) set to \(F(x)=x^k\) gives a lower bound of order \(N^\gamma,\) where
$$
\gamma=\frac{\log_m r_k(m^k)}{k}.
$$
If these bounds are the same, then comparing exponents gives \(r_k(m^k)=m^{k-1}r_k(m) \) for \(m\) square-free.
We prove the following.

\begin{proposition} \label{mod2}
\begin{enumerate}[label=\upshape(\roman*), widest=iii]
\item For \(m\) square-free, we have 
    \[r_k(m^k)\geq m^{k-1}r_k(m).\] 
\item For \(p\) prime not divisible by \(k,\) we have
\[r_k(p^k)=p^{k-1}r_k(p).\] 
    \end{enumerate}
\end{proposition}

\begin{proof}[Proof of Proposition \ref{mod2} (i)]
Let \(R \subseteq \{0,1,\dots,m-1\}\) witness \(r_k(m).\) Once again, consider the set of integers \(A\) with \(m\)-ary expansion \(\sum_{0\leq i<k}u_im^i,\) such that \(u_0\in R\) and the remaining \(u_i\in \{0,1 ,\dots, m-1\}\) are arbitrary. We may verify that this set has no differences equal to a non-zero \(k\)th power modulo \(m^k\) in a way similar to our previous arguments: explicitly, for elements \(u=\sum u_im^i\) and \(v=\sum v_im^i\) in \(A,\) suppose for contradiction that \(u-v\equiv x^k \not \equiv 0 \pmod {m^k}.\) 

This reduces to \(u_0-v_0 \equiv x^k \pmod m,\) which contradicts the definition of \(R\) if \(u_0\neq v_0.\) Thus \(u_0=v_0\) and so \(u-v= \sum_{i\geq 1} (u_i-v_i)m^i \equiv 0 \pmod m.\) Therefore \(x^k \equiv 0 \pmod m,\) which leads to \(x^k \equiv 0 \pmod {m^k},\) as \(m\) is square-free. This is a contradiction. Finally, by the definition of \(r_k(m^k),\) we have that \(r_k(m^k) \geq \vert A \vert = m^{k-1}r_k(m).\) 
\end{proof}

The inequality in part (i) can fail if \(m\) is not square-free. For example, \(r_2(16)=6\)  is strictly less than \(4r_2(4)=8\).  In order to prove part (ii) we utilise a method to `lift' solutions \cite[\S~2.6]{Intro}.

\begin{lemma}[Hensel's Lemma]
Let \(p\) be prime and \(f\in \mathbb{Z}[x].\) Suppose there exists \(a\in\mathbb{Z}\) such that \(f(a)\equiv 0 \pmod p\) and \(f'(a)\not\equiv 0 \pmod p.\) Then for every positive integer \(N\) there exists an integer \(a_N\equiv a \pmod p\) such that \(f(a_N)\equiv 0 \pmod {p^N}.\)
\end{lemma}

\begin{corollary}\label{henscor}
Let \(w\in \mathbb{Z}.\) Suppose \(p\nmid k\) and \(p\nmid w.\) Then \(w\) is a non-zero \(k\)th power  modulo \(p^N\) for all \(N\) if and only if it is a non-zero \(k\)th power  modulo \(p.\)
\end{corollary}
\begin{proof}
The forward direction is immediate. For the reverse direction simply consider \(f(x)=x^k-w.\)
\end{proof}

\begin{proof}[Proof of Proposition \ref{mod2} (ii)]
Let the set $R'$ witness \(r_k(p^k)\). We can write this in the form \[R'=\left\{up+r:(u,r)\in S \right\},\] for some \(S \subseteq \{0,1,\dots, p^{k-1}-1\} \times \{0,1,\dots, p-1\}\).   We claim that we have the fibre bounds \(\vert\{(u_0,r)\in S\}\vert\leq r_k(p)\) for each \(u_0 \in \{0,1,\dots,p^{k-1}-1\},\) which would mean \(\vert R'\vert\leq r_k(p)p^{k-1}\) and prove the result. Suppose the opposite; then there must exist distinct \((u_0,r),\left(u_0,r'\right)\in S\) with \(r-r'\) a non-zero \(k\)th power modulo \(p.\) By Corollary \ref{henscor} we therefore deduce that \(r-r'\) is a non-zero \(k\)th power modulo \(p^k.\)
However, by the definition of \(R'\) we have that for all \(x,\)
\[(u_0p+r)-\left(u_0p+r'\right)\not\equiv x^k \pmod {p^k}.\]
That is, \begin{equation*}\label{contr}
r-r'\not\equiv x^k \pmod {p^k}.
\end{equation*} 
This is a contradiction. 
\end{proof}
\section{An Open Problem}\label{sectionopen}
In Theorem \ref{polygeneralised} we required conditions on the polynomial \(f.\) Can these conditions be relaxed? For instance, is there a way to handle polynomials with lower order terms than \(x^2\) in general, such as \(f(x)=x^2+x\) or \(f(x)=x^2+2x\)? The latter of these is equivalent to \(x^2-1\) by a linear shift \(x \mapsto x-1.\)
\bibliography{bib}
\bibliographystyle{amsalpha}
\end{document}